\documentclass[11pt]{article}
\usepackage{amsmath,amsthm, amssymb, amsfonts,latexsym,amscd}
\usepackage[dvips]{graphicx}
\newtheorem{theorem}{Theorem}[section]
\newtheorem{corollary}[theorem]{Corollary}
\newtheorem{lemma}[theorem]{Lemma}
\newtheorem{proposition}[theorem]{Proposition}

\newtheorem{remark}[theorem]{Remark}
\newtheorem{example}[theorem]{Example}

\begin{document}

\begin{center}
{\Large\bf Center, centroid and subtree core of trees }\\
\vskip 1cm

{\large Dheer Noal Sunil Desai\hskip 1cm Kamal Lochan Patra}
\end{center}

\vskip 1cm

\begin{quote}
{\bf Abstract.} For $n\geq 5$ and  $2\leq g\leq n-3,$ consider the tree $P_{n-g,g}$ on $n$ vertices  which is obtained by adding $g$ pendant vertices to one degree $1$ vertex of the path $P_{n-g}$. We call the trees $P_{n-g,g}$ as path-star trees. We prove that over all trees on $n\geq 5$ vertices, the distance between center and  subtree core and the distance between centroid and subtree core are maximized by some path-star trees. We also prove that the tree $P_{n-g_0,g_0}$ maximizes both the distances among all path-star trees on $n$ vertices, where $g_0$ is the smallest positive integer such that $2^{g_0}+g_0>n-1.$ \\

\noindent {\bf Keywords}: Tree; Center; Centroid; Subtree core; Distance.\\

\noindent {\bf AMS subject classifications}. 05C05
\end{quote}

\section{Introduction}
Let $T$ be a tree with vertex set $V=V(T)$ and edge set $E=E(T)$. We denote by $d(v)$ the degree of a vertex $v\in V$. A vertex of degree one is called a {\it pendant} vertex of $T$. For $u,v\in V,$ the {\it length} of the $u-v$ path in $T$ is the number of edges in that path, and the {\it distance} between $u$ and $v$ in $T$, denoted by $d_T(u,v)$, is the length of the $u-v$ path. For subsets $U$ and $W$ of $V$, the distance $d_T(U,W)$ between $U$ and $W$ is defined by
$$d_T(U,W)=\underset{u\in U,w\in W}{\min}d_T(u,w).$$
For $v\in V$, the \textit{eccentricity} $e(v)$ of $v$ is defined by $e(v)=\max\{d_T(u,v):u\in V\}$. The \textit{radius} $rad(T)$ of $T$ is defined by $rad(T)=\min\{e(v):v\in V\}$
and the \textit{diameter} $diam(T)$ of $T$ is defined by $diam(T)=\max\{e(v):v\in V\}$. It is clear that $diam(T)=\max\{d_T(u,v):u,v\in V\}$. We say that $v$ is a \textit{central vertex} of $T$ if $e(v)=rad(T).$ The \textit{center} of $T,$ denoted by $C=C(T),$ is the set of all central vertices of $T$.

In a tree $T$, for any vertex $v$, $d_T(u,v)$ is maximum only when $u$ is a pendant vertex. Using this observation, the following result is proved (see \cite[Theorem 4.2]{har}).

\begin{theorem}\label{thm-center}
The center of a tree consists of either one vertex or two adjacent vertices.
\end{theorem}

From the proof of the above result as given in \cite[Theorem 4.2]{har}, it is clear that, for any tree $T$, $C(T)$ is same as the center of any $u-v$ path in $T$ of length $diam(T)$.

For $v\in V$, a \textit{branch} (rooted) at $v$ is a maximal subtree containing $v$ as a pendant vertex. Note that the number of branches at $v$ is $d(v)$. The \textit{weight} of $v$, denoted by $\omega(v)=\omega_T(v)$, is the maximal number of edges in any branch at $v.$ We say that $v$ is a \textit{centroid vertex} of $T$ if $\omega(v)= \underset{u\in V}{\min} \;\omega(u).$ The \textit{centroid} of $T,$ denoted by $C_d=C_d(T)$, is the set of all centroid vertices of $T$.

The following result for the centroid of a tree is analogous to Theorem \ref{thm-center} (see \cite[Theorem 4.3]{har}).

\begin{theorem}\label{thm-centroid}
The centroid of a tree consists of either one vertex or two adjacent vertices.
\end{theorem}

Let $T$ be a tree on $n$ vertices. If $|C_d(T)|=2$ and $C_d(T)=\{u,v\}$, then $n$ must be even and $\omega (u)=\omega(v)=\frac{n}{2}.$ Also, among the branches at $u$ (respectively, at $v$), the branch containing $v$ (respectively, $u$) has the maximum number of edges. If $n\geq 3$, then observe that neither the center nor the centroid of $T$ contain pendant vertices.
In general, there is no relation between the center and the centroid of a tree with regard to the number of vertices or to their location.

Like center and centroid, many researchers have defined middle part of a tree in several other ways. In \cite{zel}, Zelinka defined the notion `median' and proved that it coincides with the centroid for a tree. In \cite{mit}, Mitchell defined the `telephone center' of a tree and proved that it also coincides with the centroid. In 2005, Szekely and Wang defined in \cite{sw} a new middle part, the so called `subtree core', of a tree which does not coincide with either the center or the centroid in general.

Let $\mathbb{N}$ be the set of natural numbers. For a given tree $T$, let $f_T:V\rightarrow\mathbb{N}$ be the function defined by $v\mapsto f_T(v)$, where $f_T(v)$ is the number of subtrees of $T$ containing $v$. Then the \textit{subtree core} of $T$, denoted by $S_c=S_c(T)$, is defined as the set of all vertices $v$ for which $f_T(v)$ is maximum.

In the spirit of Theorems \ref{thm-center} and \ref{thm-centroid}, Szekely and Wang proved the following result in \cite[Theorem 9.1]{sw}.

\begin{theorem}\label{thm-score}\cite{sw}
The subtree core of a tree consists of either one vertex or two adjacent vertices.
\end{theorem}

While proving the above theorem the authors used the fact that the function $f_T$ is strictly concave in the following sense.

\begin{lemma}\label{concave}
If $u,v,w$ are three vertices of a tree $T$ with $\{u,v\},\{v,w\}\in E(T)$, then $2f_T(v)-f_T(u)-f_T(w)>0.$
\end{lemma}

We shall use the above lemma frequently, mostly without mention.

\begin{remark}\label{rem-psc}
Like the center and centroid, for any tree $T$ on $n\geq 3$ vertices, $S_c(T)$ does not contain any pendant vertex.
\end{remark}

This remark can be seen as follows. Let $v$ be a pendant vertex of $T$ and let $\{v,w\}\in E(T)$. Consider the tree $T'=T-v.$ There is only one subtree of $T$, namely $\{v\}$, containing $v$ but not $w$. The number of subtrees of $T$ containing both $v$ and $w$ is equal to $f_{T'}(w)$. So $f_T(v)=1+f_{T'}(w)$. A similar argument gives $f_T(w)=2f_{T'}(w)$. Since $n\geq 3$, we have $f_{T'}(w)\geq 2$ and hence $f_T(w)>f_T(v)$.

We denote by $P_n$ the path on the $n$ vertices $1,2,\cdots,n$, where $1$ and $n$ are pendant vertices, and for $i=2,3,\cdots, n-1,$ vertex $i$ is adjacent to vertices $i-1$ and $i+1$. The center, centroid and subtree core coincide for a path. More precisely, for $n=2m$, we have $$C(P_{2m})=C_d(P_{2m})=S_c(P_{2m})=\{m,m+1\}$$
and for $n=2m+1$, we have
$$C(P_{2m+1})=C_d(P_{2m+1})=S_c(P_{2m+1})=\{m+1\}.$$
We denote by $K_{1,n-1}$ the star on the $n$ vertices $1,2,\cdots,n$, where $n$ is the only non-pendant vertex. Then
$$C(K_{1,n-1})=C_d(K_{1,n-1})=S_c(K_{1,n-1})=\{n\}.$$
We now give an example of a tree in which the center, centroid and subtree core are pair-wise different.

\begin{figure}[!ht]
 \includegraphics[scale=1]{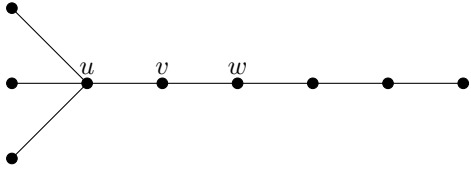}
\caption{Tree with different center, centroid and subtree core}\label{fig:1}
\end{figure}

\begin{example}\label{exp-ccs}
Consider the tree $T$ on 9 vertices as in Figure \ref{fig:1}. We have $e(w)=3$ and the eccentricity of any vertex other than $w$ is at least $4.$ So $C(T)=\{w\}.$ We have $\omega(v)=4$ and the weight of any vertex other than $v$  is at least $5$. So $C_d(T)=\{v\}.$ Finally, $f_T(u)=48$ and $f_T(x)<48$ for any vertex $x$ other than $u$. So $S_c(T)=\{u\}.$
\end{example}

For a given tree $T$, we denote by $d_T(C,C_d)$ (respectively, $d_T(C_d,S_c)$, $d_T(C,S_c)$) the distance between the center and the centroid (respectively, the centroid and the subtree core, the center and the subtree core) of $T$. It is clear that the minimum of $d_T(C,C_d)$ (respectively, $d_T(C_d,S_c)$, $d_T(C,S_c)$) among all trees $T$ on $n$ vertices is zero. The maximum of $d_T(C,C_d)$ among all trees $T$ on $n$ vertices has been studied in \cite{klp}, which we describe below.

\subsection{Path-star trees}

Let $P_{n-g,g}, \; n\geq 2, \; 1\leq g \leq n-1$, denote the tree on $n$ vertices which is obtained from the path $P_{n-g}$ by adding $g$ pendant vertices to the vertex $n-g$ (see
Figure \ref{fig:2}). Such a tree $P_{n-g,g}$ is called a {\em path-star tree}.
\begin{figure}[!ht]
 \includegraphics[scale=0.7]{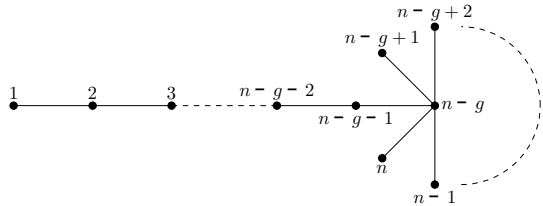}
\caption{Path-star tree}\label{fig:2}
\end{figure}

Note that $P_{n-1,1}$ is a path, and $P_{1,n-1}$ and $P_{2,n-2}$ are stars.  Any tree on less than or equal to $4$ vertices is a star or a path. The
exact location of the center, the centroid and the subtree core of paths and stars have already been mentioned. Therefore, for a part-star tree, we assume throughout that
$$n\geq 5 \text{ and } 2\leq g\leq n-3.$$
We denote by $\Gamma_n$ the class of all path-star trees $P_{n-g,g}$ with the above restrictions on $n$ and $g$. Then $|\Gamma_n|=n-4.$
For the tree $P_{3,n-3}$, we have $C(P_{3,n-3})=\{2,3\}$ and $C_d(P_{3,n-3})=\{3\}=S_c(P_{3,n-3})$. Hence
$$d_{P_{3,n-3}}(C,C_d)=d_{P_{3,n-3}}(C_d,S_c)=d_{P_{3,n-3}}(C,S_c)=0.$$
Any tree $T_5$ on $5$ vertices is either a path, or a star, or isomorphic to $P_{3,2}$. Therefore, $d_{T_5}(C,C_d)=d_{T_5}(C_d,S_c)=d_{T_5}(C,S_c)=0.$

In \cite[Theorems 2.3, 3.5]{klp}, the following results are obtained regarding the maximum distance between the center and the centroid among all trees on $n$ vertices.

\begin{theorem}\label{thm:cc1}\cite{klp}
Among all trees in $\Gamma_n$, the distance between the center and the centroid is maximized when $g=\left\lfloor\frac{n}{2}\right\rfloor$. If $T$ is a tree on $n\geq 5$ vertices, then
\begin{enumerate}
\item[(1)] $d_{P_{n-g,g}}(C,C_d)\geq d_T(C,C_d)$ for some $2\leq g\leq n-3$.
\item[(2)] $d_T(C,C_d)\leq \left\lfloor \frac{n-3}{4}\right\rfloor.$
\end{enumerate}
\end{theorem}

In this paper, we study the problem of maximizing the distances $d_T(C,S_c)$ and $d_T(C_d,S_c)$ among all trees $T$ on $n$ vertices, in which path-star trees would also play an important role. More precisely, we prove the following.

\begin{theorem}\label{main-1}
Let $T$ be a tree on $n\geq 5$ vertices and let $g_0$ be the smallest positive integer such that $2^{g_0} + g _0> n - 1$. Then 
\begin{enumerate}
\item[(i)] $d_T(C,S_c)\leq \lfloor \frac{n-g_0}{2}\rfloor -1$.
\item[(ii)] $d_T(C_d,S_c)\leq \lfloor \frac{n-1}{2} \rfloor -g_o.$
\end{enumerate}
Further, these bounds are attained by the path-star tree $P_{n-g_0,g_0}$.
\end{theorem}

\section{Center and subtree core}

For given vertices $v_1,v_2,\cdots,v_k$ in a tree $T$, we denote by $f_T(v_1,v_2,\cdots,v_k)$ the number of subtrees of $T$ containing $v_1,v_2,\cdots,v_k$.

\begin{lemma}\label{lem-elem}
Let $T$ be a tree and $w,y\in V(T)$, where $y$ is a pendant vertex not adjacent to $w$. Let $\widetilde{T}$ be the tree obtained by detaching $y$ from $T$ and adding it as a pendant vertex adjacent to $w$. Then
$f_{\widetilde{T}}(a)=f_T(a)-f_{T}(a,y)+f_T(a,w)-f_{T}(a,w,y)$ for any $a\in V(T-y)$.
\end{lemma}

\begin{proof}
Write $T'=T-y=\widetilde{T}-y$. Observe that, for any $a\in V(T')$, the number of subtrees of $\widetilde{T}$ containing $a$ but not $y$ is equal to $f_{T'}(a)$, and the number of subtrees of $\widetilde{T}$ containing both $a$ and $y$ is equal to $f_{T'}(a,w)$. So
\begin{equation}\label{eq1}
f_{\widetilde{T}}(a)=f_{T'}{(a)} + f_{T'}(a,w).
\end{equation}
The set of subtrees of $T$ containing $a$ is a disjoint union of the subtrees of $T$ containing $a$ but not $y$, and the subtrees of $T$ containing both $a$ and $y$. This gives
\begin{equation}\label{eq2}
f_{T'}(a)=f_T{(a)}-f_{T}{(a,y)},
\end{equation}
since $f_{T'}(a)$ is equal to the number of subtrees of $T$ containing $a$ but not $y$. Similarly, we get
\begin{equation}\label{eq3}
f_{T'}(a,w)=f_T{(a,w)}-f_{T}{(a,w,y)}.
\end{equation}
Then $f_{\widetilde{T}}(a)=f_T(a)-f_{T}(a,y)+f_T(a,w)-f_{T}(a,w,y)$, follows from (\ref{eq1}), (\ref{eq2}) and (\ref{eq3}).
\end{proof}

The following lemma compares the subtree core of two trees when one is obtained from the other by some graph perturbation.

\begin{lemma}\label{lem-csc1}
Let  $T$ be a tree, $v\in S_c(T)$ and $y$ be a pendant vertex of $T$ not adjacent to $v$. If $\widetilde{T}$ is the tree obtained by detaching $y$ from $T$ and adding it as a pendant vertex adjacent to $v$, then $S_c(\widetilde{T})=\{v\}$.
\end{lemma}

\begin{proof}
We show that $f_{\widetilde{T}}(v) > f_{\widetilde{T}}(b)$ for every $b\in V(\widetilde{T})\setminus\{v\}$. Let $u_1,u_2,\cdots,u_{k+1}=y$ be the vertices adjacent to $v$ in $\widetilde{T}$. As $y$ is a pendant vertex, Remark \ref{rem-psc} implies $y\notin S_c(\widetilde{T})$. Again, by Lemma \ref{concave}, it is enough to show that $f_{\widetilde{T}}(v) - f_{\widetilde{T}}(u_i) > 0$ for $1\leq i\leq k.$
Taking $a=u_i$ and $w=v$ in Lemma \ref{lem-elem}, we have
$$f_{\widetilde{T}}(u_i)=f_T(u_i)-f_{T}(u_i,y)+f_T(u_i,v)-f_{T}(u_i,v,y)$$
Again, taking $a=w=v$ in Lemma \ref{lem-elem}, we have
$$ f_{\widetilde{T}}(v)=f_T(v)-f_{T}(v,y)+f_T(v)-f_{T}(v,y).$$
Then, for $1\leq i\leq k$, we have
\begin{equation*}
f_{\widetilde{T}}(v) - f_{\widetilde{T}}(u_i) =[f_T(v)-f_T(u_i)]+f_T(v)-2f_T(v,y)+f_T(u_i,y)-f_T(u_i,v)+f_T(u_i,v,y).
\end{equation*}
Since $v\in S_c(T)$, $f_T(v)-f_T(u_i)\geq 0$. Note that here equality may happen for one $i$ if $S_c(T)=\{u_i,v\}$. Thus, it is enough to prove that
$$f_T(v)-2f_T(v,y)+f_T(u_i,y)-f_T(u_i,v)+f_T(u_i,v,y)>0,$$
for $i=1,2,\cdots,k.$

Let $A_i$ and  $B_i$ be the components of $T$ obtained by deleting the edge $\{v,u_i\}\in E(T)$. We may assume that $A_i$ contains $v$ and $B_i$ contains $u_i$.
Then
\begin{align}
f_T(v) & = f_{A_i}(v) + f_{A_i}(v)f_{B_i}(u_i)\label{eq-c-1}\\
f_T(u_i,v) & = f_{A_i}(v)f_{B_i}(u_i)\label{eq-c-2}
\end{align}
We shall consider two cases depending on whether $y\in A_i$ or $y\in B_i$.\\

\noindent\underline{Case 1:} $y\in A_i$. Here we have the following.
\begin{align*}
f_T(v,y) & = f_{A_i}(v,y) + f_{A_i}(v,y)f_{B_i}(u_i)\\
f_T(u_i,y) & = f_{A_i}(v,y)f_{B_i}(u_i)\\
f_T(u_i,v, y) & = f_{A_i}(v,y)f_{B_i}(u_i)
\end{align*}
Using the above three equations together with (\ref{eq-c-1}) and (\ref{eq-c-2}), we get
$$f_T(v)-2f_T(v,y)+f_T(u_i,y)-f_T(u_i,v)+f_T(u_i,v,y) = f_{A_i}(v) - 2f_{A_i}(v,y).$$
Let $y'$ be the (unique) vertex adjacent to $y$ in $T$. Then $y'\in A_i$ and $2f_{A_i}(v,y) = f_{A_i}(v,y').$ Therefore,
$f_{A_i}(v) - 2f_{A_i}(v,y)= f_{A_i}(v) - f_{A_i}(v,y')>0$, as $v\neq y'$.\\

\noindent\underline{Case 2:} $y\in B_i$. In this case, we have
\begin{align*}
f_T(v,y) & = f_{A_i}(v)f_{B_i}(u_i,y)\\
f_T(u_i,y) & = f_{B_i}(u_i,y) + f_{A_i}(v)f_{B_i}(u_i,y)\\
f_T(u_i,v, y) & = f_{A_i}(v)f_{B_i}(u_i,y)
\end{align*}
Using the above three equations together with (\ref{eq-c-1}) and (\ref{eq-c-2}), we get
$$f_T(v)-2f_T(v,y)+f_T(u_i,y)-f_T(u_i,v)+f_T(u_i,v,y) = f_{A_i}{(v)} + f_{B_i}{(u_i, y)} > 0.$$
This completes the proof.
\end{proof}

We now prove the following result which says that, among all trees on $n$ vertices, the distance between the center and the subtree core is maximized by a path-star tree.

\begin{theorem}\label{thm-csc}
Let $T$ be any tree on $n\geq 5$ vertices. Then there exists a path-star tree $P_{n-g,g}$, for some $g$, with $d_{P_{n-g,g}}(C, S_c)\geq d_T(C, S_c)$.
\end{theorem}

\begin{proof}
We may assume that $d_T(C, S_c)\geq 1$. Let $C(T)=\{w_1, w_2\}$ and $S_c(T) = \{v_1, v_2\}$, where $w_1=w_2$ if $|C(T)|=1$ and $v_1=v_2$ if $|S_c(T)|=1$. We may also assume that $d_T(C, S_c)=d_T(w_2, v_1)$.

Let $B_1, B_2, \cdots, B_m$ be the branches of $T$ at $v_2$. Without loss, we may assume that $C(T), S_c(T) \subseteq V(B_1)$. For $ i\in\{2,3,\ldots,m\}$, if there are pendant vertices $y$ of $T$ contained in $B_i$ but not adjacent to $v_2$, then detach $y$ from $T$ and add it as a pendant vertex adjacent to $v_2$. It may happen that a given non-pendant vertex $x$ in $B_i$ becomes a pendant after deletion of certain pendant vertices, apply the same procedure to $x$ as well (we shall use this graph operation more frequently). Continue this process till all the vertices of $T$, not in $B_1$, are attached to $v_2$ as pendants. We denote by $\widetilde{T}$ the new tree obtained from $T$ in this way. By Lemma \ref{lem-csc1},  $S_c(\widetilde{T}) = \{v_2\}$. If the vertices of $T$, not in $B_1$, are all pendants, then we take $\widetilde{T}=T$ and proceed with $S_c(\widetilde{T})= \{v_1, v_2\}$.

We now study the position of $C(\widetilde{T}).$ We know that the center of a tree is the center of any longest path in it. If there is a longest path in $T$ which does not contain $v_2,$ then that path is contained in $B_1$. In that case, $diam(T)=diam(\widetilde{T})$ and hence $C(T)=C(\widetilde{T})$. Otherwise, every longest path in $T$ contains $v_2$. Then $diam(T)\geq diam(\widetilde{T})$ and so $C(\widetilde{T})$ may move away from $S_c(T)$ with respect to a path in $\widetilde{T}$ containing $C(T)$ and $v_2$. Therefore, $d_{\widetilde{T}}(C,S_c)\geq d_T(C,S_c)$.

If $\widetilde{T}$ is a path-star tree, then we are done. Otherwise, let $P$ be a longest path in $B_1$ containing both $C(\widetilde{T})$ and $S_c(\widetilde{T})$. We now proceed by detaching pendant vertices of $B_1$ but not in $P$, and add them as pendant vertices adjacent to $v_2$, one after another, till we are left with only the path $P$ in $B_1$. Call the new tree obtained in this way from $\widetilde{T}$ as $\overline{T}$. By Lemma \ref{lem-csc1}, $S_c(\overline{T}) = \{v_2\}$. Clearly, $\overline{T}$ is a path-star tree. Since $diam(\overline{T}) \leq diam(\widetilde{T})$, it follows that $d_{\overline{T}}(C,S_c) \geq d_{\widetilde{T}}(C,S_c)$. Thus, we have a path-star tree $\overline{T}$ on $n$ vertices with $d_{\overline{T}}(C, S_c) \geq d_{\widetilde{T}}(C, S_c) \geq d_T(C,S_c)$.
\end{proof}

We next try to find a relation between $n$ and $g$ for which the distance $d_{P_{n- g, g}}(C, S_c)$ is maximum. We first look for the position of the subtree core in any $P_{n - g, g}.$
For $1\leq i\leq n-g,$ we have
\begin{equation}\label{eq20}
f_{P_{n - g, g}}(i) = i(n-g-i) + i(2^g).
\end{equation}
Here the first term denotes the number of subtrees of $P_{n-g,g}$ containing the vertex $i$ but not $n-g$, while the second term counts the number of subtrees of $P_{n-g,g}$ containing both $i$ and $n-g$.
For $n-g+1\leq i\leq n,$  we have
\begin{equation}\label{eq21}
f_{P_{n - g, g}}(i) = 1 + (n-g)(2^{g-1}).
\end{equation}
Here $1$ accounts for the number of subtrees of $P_{n-g,g}$ containing $i$ but not $n-g$, while the second term is the number of subtrees of $P_{n-g,g}$  containing both $i$ and $n-g$.

The subtree core of $P_{n-g,g}$ lies in the path from $2$ to $n-g$, as it does not contain any pendant vertex. We have $S_c(P_{n-g,g}) = \{n - g\}$ if and only if $f_{P_{n-g,g}}(n-g) - f_{P_{n-g,g}}(n - g - 1) > 0.$ Using equation (\ref{eq20}), $f_{P_{n-g,g}}(n-g) - f_{P_{n-g,g}}(n - g - 1)=(n - g)(2^g) - (n - g - 1)(1) - (n - g - 1)(2^g)=g - n + 1 + 2^g.$ Therefore,
$$S_c(P_{n-g,g}) = \{n - g\}\text{ if and only is } 2^g + 1 > n - g.$$

Now suppose that $2^g+1\leq n-g$. Then the subtree core of $P_{n-g,g}$ intersects the path connecting the vertex $2$ to $n-g-1$. So, for $j \in \{2, 3, \ldots, n-g-1\}$ with $j \in S_c(P_{n-g,g})$ and $j-1 \notin S_c(P_{n-g,g})$, we have $f_{P_{n-g,g}}(j) - f_{P_{n-g,g}}(j-1)>0$ and $f_{P_{n-g,g}}(j+1) - f_{P_{n-g,g}}(j)\leq 0$ (with equality if and only if $j+1\in S_c(P_{n-g,g}))$. By using equation (\ref{eq20}), we have
$f_{P_{n-g,g}}(j) - f_{P_{n-g,g}}(j-1) = n - g - 2j + 1 + 2^g > 0.$ So, $j < \frac{n - g + 1 + 2^g}{2}.$
Also,
\begin{equation*}
\begin{aligned}
&f_{P_{n-g,g}}(j+1) - f_{P_{n-g,g}}(j) \\
= &(j+1)(n-g-j-1) + (j+1)(2^g) -j(n-g-j) - j(2^g)\\
= &  n - g - 2j - 1 + 2^g.
\end{aligned}
\end{equation*}
If $j+1 \notin S_c(P_{n-g,g})$, then $n - g - 2j - 1 + 2^g < 0$ and so $j > \frac{n - g - 1 + 2^g}{2}.$ Then
$$\frac{n - g - 1 + 2^g}{2} < j < \frac{n - g + 1 + 2^g}{2}$$
gives $S_c(P_{n-g,g}) = \{j\} = \{\frac{n - g + 2^g}{2}\}.$ It follows that $n - g$ must be even. Now, if $j+1 \in S_c(P_{n-g,g})$, then $n - g - 2j - 1 + 2^g = 0$ and so $j = \frac{n - g - 1 + 2^g}{2}.$ Therefore, $S_c(P_{n-g,g}) = \{j, j+1\}$, where $j = \frac{n - g - 1 + 2^g}{2}.$ In this case, $n - g$ must be odd. Thus we have the following.

\begin{theorem}\label{remark1}
The subtree core of the path-star tree $P_{n-g, g}$ is given by
\begin{equation*}
S_c(P_{n-g, g}) =
\begin{cases}
\begin{cases}
\left\{\frac{n - g + 2^g}{2}\right\}, &\text{if $n - g$ is even}\\
\left\{\frac{n - g - 1 + 2^g}{2}, \frac{n - g + 1 + 2^g}{2}\right\} , &\text{if $n - g$ is odd}
\end{cases}, &\text{if $2^g + 1 \leq n - g$,}\\
\{n - g\}, &\text{if $2^g + 1 > n - g.$}
\end{cases}
\end{equation*}
\end{theorem}

The position of the center of $P_{n-g,g}$ can also be expressed in terms of $n-g$. The following result is straight-forward.

\begin{theorem}\label{remark2}
The center of the path-star tree $P_{n-g, g}$ is given by
\begin{equation*}
C(P_{n - g, g}) =
\begin{cases}
\left\{\frac{n - g + 2}{2}\right\}, &\text{if $n - g$ is even,}\\
\left\{\frac{n - g + 1}{2}, \frac{n - g + 3}{2}\right\}, &\text{if $n - g$ is odd.}
\end{cases}
\end{equation*}
\end{theorem}

\begin{theorem}\label{remark3}
The distance between the center and the subtree core of the path-star tree $P_{n-g,g}$ is given by
\begin{equation*}\label{eq22}
d_{P_{n - g, g}}(C, S_c) =
\begin{cases}
\begin{cases}
2^{g - 1} - 1, &\text{if $n - g$ is even}\\
2^{g - 1} - 2, &\text{if $n - g$ is odd}
\end{cases}, &\text{if $2^g + 1 \leq n - g$}\\
\begin{cases}
\frac{n - g - 2}{2}, &\text{if $n - g$ is even}\\
\frac{n - g - 3}{2}, &\text{if $n - g$ is odd}\\
\end{cases}, &\text{if $2^g + 1 > n - g.$}
\end{cases}
\end{equation*}
\end{theorem}

\begin{proof}
From Theorems \ref{remark1} and \ref{remark2}, we have
\begin{equation*}
d_{P_{n - g, g}}(C, S_c) =
\begin{cases}
\begin{cases}
\frac{n - g + 2^g - n + g - 2}{2}, &\text{if $n - g$ is even}\\
\frac{n - g - 1 + 2^g - n + g -3}{2}, &\text{if $n - g$ is odd}
\end{cases}, &\text{if $2^g + 1 \leq n - g$}\\
\begin{cases}
\frac{2n - 2g - n + g - 2}{2}, &\text{if $n - g$ is even}\\
\frac{2n - 2g - n + g - 3}{2}, &\text{if $n - g$ is odd}\\
\end{cases}, &\text{if $2^g + 1 > n - g.$}
\end{cases}
\end{equation*}
Now the result follows from the above.
\end{proof}

For a given $n\geq 5$, we now try to find a $g \in \{ 2, \cdots, n - 3\}$ for which $d_{P_{n - g, g}}(C, S_c)$ is maximum among all trees on $n$ vertices. Let $g_0$ denote the smallest value of $g$, with respect to $n$, for which $2^{g_0} + g_{0} > n -1$ (that is, $2^{g_0} + 1 > n - g_0$). By Theorem \ref{remark1}, $S_c(P_{n - g_0, g_0}) = \{n - g_0\}$. Then, by Theorem \ref{remark3}, we have
\begin{equation}\label{eq23}
d_{P_{n - g_0, g_0}}(C, S_c) =
\begin{cases}
\frac{n - g_0 - 2}{2}, &\text{if $n - g_0$ is even}\\
\frac{n - g_0 - 3}{2}, &\text{if $n - g_0$ is odd.}
\end{cases}
\end{equation}

\begin{proposition}\label{propn3}
Among all trees on $n\geq 5$ vertices, the path-star tree $P_{n-g_0,g_0}$ maximizes the distance between the center and the subtree core.
\end{proposition}

\begin{proof}
By Theorem \ref{thm-csc}, we need to prove the following:
$$d_{P_{n - g_0, g_0}}(C, S_c) \geq d_{P_{n - g_0 + k, g_0 - k}}(C,S_c)\text{ for }k \in \{1, 2, \ldots, g_0 - 2\}$$
and
$$d_{P_{n - g_0, g_0}}(C, S_c) \geq d_{P_{n - g_0 - l, g_0 + l}}(C, S_c)\text{ for }l \in \{1, 2, \ldots, n - g_0 - 3\}.$$
First assume that $n - g_0$ is odd. Then, for $k \in \{1, 2, \ldots, g_0 - 2\}$, we have
\begin{equation}
\label{eq39}
\begin{aligned}
d_{P_{n - g_0, g_0}}(C, S_c) - d_{P_{n - g_0 + k, g_0 - k}}(C, S_c) &\geq \frac{n - g_0 - 3}{2} - 2^{g_0 - k - 1} + 1\\
&\geq \frac{n - g_0 - 3}{2} - 2^{g_0 - 1 - 1} + 1\\
&= \frac{n - g_0 - 3 - 2^{g_0 - 1} + 2}{2}\\
&= \frac{n - g_0 - 1 - 2^{g_0 - 1}}{2}.
\end{aligned}
\end{equation}
By the definition of $g_0$, we have $2^{g_0-1}\leq n-(g_0-1)-1=n-g_0.$ Therefore,
$$d_{P_{n - g_0, g_0}}(C, S_c) - d_{P_{n - g_0 + k, g_0 - k}}(C, S_c) \geq \frac{n - g_0 - 1 - 2^{g_0 - 1}}{2}
\geq \frac{-1}{2}.$$
Since both $d_{P_{n - g_0, g_0}}(C, S_c)$ and $d_{P_{n - g_0 + k, g_0 - k}}(C, S_c)$  are integers, their difference must be an integer. Hence
$$d_{P_{n - g_0, g_0}}(C, S_c) - d_{P_{n - g_0 + k, g_0 - k}}(C, S_c) \geq 0.$$
Now, for $l \in \{1, 2, \ldots, n - g_0 - 3\}$, we have
\begin{equation}\label{eq42}
\begin{aligned}
d_{P_{n - g_0, g_0}}(C, S_c) - d_{P_{n - g_0 - l, g_0 + l}}(C, S_c) &\geq \frac{n - g_0 - 3}{2} -\left(\frac{n - g_0 - l - 2}{2}\right)\\
&\geq \frac{n - g_0 - 3}{2} -\left(\frac{n - g_0 - 1 - 2}{2}\right)\\
&= 0.
\end{aligned}
\end{equation}
For the case $n-g_0$ even, the above arguments can be used again as $\frac{n - g_0 - 2}{2} > \frac{n - g_0 - 3}{2}$. Thus,
$$d_{P_{n - g_0, g_0}}(C, S_c) \geq d_{P_{n - g, g}}(C, S_c)$$
for all $g \in \{1, 2, \ldots, n-3\}.$ This completes the proof.
\end{proof}

As a consequence of Proposition \ref{propn3}, we have the following.

\begin{corollary}
Let $T$ be tree on $n\geq 5$ vertices and let $g_0$ be the smallest positive integer such that $2^{g_0} + g _0> n - 1.$ Then $d_T(C,S_c)\leq \lfloor \frac{n-g_0}{2}\rfloor -1$.
\end{corollary}

\section{Centroid and subtree core}

In this section we prove results similar to Theorem \ref{thm-csc} and Proposition \ref{propn3} in the case of centroid and subtree core. We first prove the following lemma.

\begin{lemma}\label{lem-cdsc}
Let $T$ be a tree, $v \in S_c(T)$ and $B$ be a branch at $v$. Let $u$ be the vertex in $B$ adjacent to $v$ and $x$ be a pendant vertex of $T$ in $B$. Suppose that $B$ is not a path. Let $y$ be the closest vertex to $x$ with  $d(y)\geq 3$ and let $[y, y_1, y_2, \cdots, y_m=x]$ be the path connecting $y$ and $x$. Let  $z\neq y$ be a vertex of $B$ such that the path from $v$ to $z$ contains $y$ but not $y_1$. Let $\widetilde{T}$ be the tree obtained from $T$ by detaching the path $[y_1, y_2, \cdots, y_m]$ from $y$ and attaching it to $z$. Then $f_{\widetilde{T}}(v) > f_{\widetilde{T}}(u)$.
\end{lemma}

\begin{proof}
Let $T'$ be the tree obtained from $T$ by removing the path $[y_1, y_2, \cdots, y_m]$. Let $a \in V(T')$. Then
\begin{equation}\label{eq45}
f_{\widetilde{T}}(a) = f_{T'}(a) + mf_{T'}(a, z).
\end{equation}
Here the first term represents the number of subtrees of $\widetilde{T}$ containing $a$ but not $y_1$, and the second term represents the number of subtrees of $\widetilde{T}$ containing both $a$ and $y_1.$
Similarly, we have $f_T(a) = f_{T'}(a) + mf_{T'}(a, y)$ and $f_T(a, y)=(m+1)f_{T'}(a,y)$. It follows that
\begin{equation}\label{eq46}
f_{T'}(a) = f_T(a) - \frac{m}{m+1}f_T(a, y).
\end{equation}
Again, $f_T(a,z) = f_{T'}(a,z) + mf_{T'}(a,z,y)$ and $f_T(a,z,y) = (m+1)f_{T'}(a,z,y)$ imply
\begin{equation}
\label{eq47}
f_{T'}(a, z) = f_T(a, z) - \frac{m}{m + 1}f_T(a, z, y).
\end{equation}
Using equations (\ref{eq46}) and (\ref{eq47}) in equation (\ref{eq45}), we get
\begin{equation}\label{eq48}
f_{\widetilde{T}}(a) = f_T(a) - \frac{m}{m+1}f_T(a, y) + m\left[f_T(a, z) - \frac{m}{m + 1}f_T(a, z, y)\right],
\end{equation}
for any $a\in V(T')$. Considering $a=v$ and $a=u$ in equation (\ref{eq48}), we get
\begin{equation}
\label{eq49}
\begin{aligned}
f_{\widetilde{T}}(v) - f_{\widetilde{T}}(u) &=  f_T(v) - \frac{m}{m+1}f_T(v, y) + mf_T(v, z) - \frac{m^2}{m + 1}f_T(v, z, y)\\
&  \;\;\;\;\;\;\;\;\; - f_T(u)+ \frac{m}{m+1}f_T(u, y) - mf_T(u, z) + \frac{m^2}{m + 1}f_T(u, z, y)\\
&\geq  \frac{m}{m + 1}\left(f_T(u, y) - f_T(v, y)\right) + m(f_T(v, z) - f_T(u, z))\\
& \;\;\;\;\;\;\;\;\; + \frac{m^2}{m + 1}(f_T(u, z, y) - f_T(v, z, y))\\
&= \frac{m}{m + 1}(f_T(u, y) - f_T(v, y)) + m(f_T(v, z) - f_T(u, z)) \\
&  \;\;\;\;\;\;\;\;\; + \frac{m^2}{m + 1}(f_T(u, z) - f_T(v, z))\\
&= \frac{m}{m + 1}(f_T(u, y) - f_T(v, y)) + \frac{m}{m + 1}(f_T(v, z) - f_T(u, z))\\
\end{aligned}
\end{equation}
In the above, the first inequality holds as $f_T(v)-f_T(u)\geq 0$ and the second last equality holds since any subtree containing $u$ and $z$ must contain $y$.
Now let $X$ be the component of $T$ containing $u$ after deleting the edge $\{u,v\}$. Then it can be seen that
$$f_T(u,y)=f_X(u,y)+f_T(v,y) \mbox{ and } f_T(u,z)=f_X(u,z)+f_T(v,z).$$
Putting these values of $f_T(u,y)$ and $f_T(u,z)$ in equation(\ref{eq49}), we get
$$f_{\widetilde{T}}(v) - f_{\widetilde{T}}(u)\geq \frac{m}{m+1}(f_X(u,y) -f_X(u,z)).$$
Since $y\neq z$, we have $f_X(u,y) > f_X(u,z)$ and so $f_{\widetilde{T}}(v) - f_{\widetilde{T}}(u) > 0$.
\end{proof}

Next, we prove a result analogous to Theorem \ref{thm-csc}. It says that, among all trees on $n$ vertices, the distance between the centroid and the subtree core is maximized by a path-star tree.

\begin{theorem}\label{thm-cdsc}
Let $T$ be any tree on $n\geq 5$ vertices. Then there exists a path-star tree $P_{n - g, g}$, for some $g$, with $d_{P_{n - g, g}}{(C_d, S_c)} \geq d_T{(C_d, S_c)}$.
\end{theorem}

\begin{proof}
We may assume that $d_T{(C_d, S_c)}\geq 1$. Let $C_d(T)=\{w_1, w_2\}$ and $S_c(T) = \{v_1, v_2\}$, where $w_1=w_2$ if $|C_d(T)|=1$ and $v_1=v_2$ if $|S_c(T)|=1$. We may also assume that $d_T{(C_d, S_c)}=d_T(w_2, v_1)$.

Let $B_1, B_2, \ldots, B_m$ be the branches at $v_2$. We assume that the branch $B_1$ contains $C_d(T)$ and $S_c(T)$. Using the same graph operations recursively as in the proof of Theorem \ref{thm-csc}, construct a new tree $\widetilde{T}$ from $T$ by attaching each vertex of $B_i$, $i\in\{2,3,\ldots,m\}$, non-adjacent with $v_2$ in $T$ as a pendant vertex adjacent to $v_2$. By Lemma \ref{lem-csc1},  $S_c(\widetilde{T}) = \{v_2\}$. If the vertices of $T$, not in $B_1$, are all pendants, then we take $\widetilde{T}=T$ and proceed with $S_c(\widetilde{T})= \{v_1, v_2\}$.

We now study the position of $C_d(\widetilde{T})$. Note that $B_1$ remains a branch in $\widetilde{T}$ at $v_2$. For any $v\in V(\widetilde{T})\setminus B_1$, $v$ is a pendant vertex in $\widetilde{T}$ and so $\omega_{\widetilde{T}}(v)=n-1\geq\omega_T(v)$. Also $\omega_{\widetilde{T}}(x)=\omega_T(x)$ for every $x\in B_1\setminus\{v_2\}$, in particular, we have
$$\omega_{\widetilde{T}}(w_1)=\omega_T(w_1)=\omega_T(w_2)= \omega_{\widetilde{T}}(w_2).$$
The weight of $v_2$ in $\widetilde{T}$ corresponds to the branch $B_1$ at $v_2$. If the weight of $v_2$ in $T$ corresponds to the branch $B_1$, then $\omega_{\widetilde{T}}(v_2)=\omega_T(v_2)>\omega_T(w_2)=\omega_{\widetilde{T}}(w_2)$. If the weight of $v_2$ in $T$ corresponds to a branch $B_j$, $j\neq 1$, then $\omega_T(w_2)$ must correspond to a branch at $w_2$ which does not contain $v_2$ and it follows that $\omega_{\widetilde{T}}(v_2)>\omega_{\widetilde{T}}(w_2)$. So $C_d(T) = C_d(\widetilde{T})$ and
$$d_{\widetilde{T}}(C_d, S_c) \geq d_{T}(C_d, S_c).$$

Now consider the path $P=[w_1,a_1,\cdots,a_k,v_2]$ from $w_1$ to $v_2$ in $\widetilde{T}$, where $a_1=w_2$ if $w_1\neq w_2$ and $a_k=v_1$ if $v_1\neq v_2$. Suppose that $d_{\widetilde{T}}(a_i)\geq 3$ for some $i$ and $B$ is a branch at $a_i$ which contains neither  $w_1$ nor $v_2$. Applying the graph operation as before (starting with the pendant vertices), attach  the vertices of $B\setminus\{a_i\}$ as pendants to $v_2$. Do this for all branches at $a_i$, $1\leq i\leq k$, with $d(a_i)\geq 3$. Name the new tree thus obtained as $\widehat{T}$. By Lemma \ref{lem-csc1},  $S_c(\widehat{T})= \{v_2\}$. If $d(a_i)=2$ for all $1\leq i\leq k$, then take $\widehat{T}=\widetilde{T}$ and continue with $S_c(\widehat{T})=S_c(\widetilde{T})$.

We now study the position of $C_d(\widehat{T})$. Any pendant vertex in $\widehat{T}$ has weight $n-1$. Also, $\omega_{\widetilde{T}}(x)=\omega_{\widehat{T}}(x)$ for any vertex $x$ in a branch at $w_1$ in $\widehat{T}$ which does not contain $v_2$. In particular, $\omega_{\widetilde{T}}(w_1)=\omega_{\widehat{T}}(w_1)$. It remains to consider the vertices $a_1,\cdots, a_k$ and $a_{k+1}=v_2$.

Case-I: The weight of $w_1$ in $\widetilde{T}$ corresponds to a branch $\bar{B}$ at $w_1$ not containing $v_2$ (in this case, observe that we must have $w_1=w_2$). The branch in $\widehat{T}$ at $a_i$ containing $w_1$ contains $\bar{B}$. So
$$\omega_{\widehat{T}}(a_i)> |E(\bar{B})|=\omega_{\widetilde{T}}(w_1)=\omega_{\widehat{T}}(w_1),$$
for all $1\leq i\leq k+1$.

Case-II: The weight of $w_1$ in $\widetilde{T}$ corresponds to the branch $\widetilde{B}$ at $w_1$ containing $v_2$. Since $\omega_{\widetilde{T}}(w_1) \leq \omega_{\widetilde{T}}(a_i)$, the weight of $a_i$ in $\widetilde{T}$ must correspond to the branch at $a_i$ containing $w_1$. Let $\widetilde{B}_1$ denote the branch in $\widetilde{T}$ at $a_1$ containing $w_1$. Note that $\widetilde{B}_1$ is also a branch at $a_1$ in $\widehat{T}$ containing $w_1$. Now, for $i=1$,
$$\omega_{\widehat{T}}(a_1)= |E(\widetilde{B}_1)|=\omega_{\widetilde{T}}(a_1)\geq \omega_{\widetilde{T}}(w_1)=\omega_{\widehat{T}}(w_1).$$
Then, for $2\leq i\leq k+1$, we have
$$\omega_{\widehat{T}}(a_i)> |E(\widetilde{B}_1)|\geq \omega_{\widehat{T}}(w_1).$$
Thus $C_d(\widetilde{T})=C_d(\widehat{T})$ and hence $d_{\widehat{T}}(C_d, S_c) \geq d_{\widetilde{T}}(C_d, S_c).$

If $\widehat{T}$ is a path-star tree, then we are done. Otherwise, let $C_1, C_2, \ldots, C_s$ be the branches at $w_1$ in $\widehat{T}$, where $C_1$ is the branch containing $v_2$. Note that, by applying continuously the process of detaching a path from a vertex of degree at least three and attaching it at a pendant vertex, we may convert any given tree into a path.

Now transform each of the branches $C_i$, $2\leq i\leq s$, into paths at $w_1$ by continuously using the graph operation as in Lemma \ref{lem-cdsc} (taking suitable pendant vertices for $z$) to obtain a new tree $\overline{T}$. The weight of $w_j$, $j\in\{1,2\}$, in $\overline{T}$ is equal to that of in $\widehat{T}$. Applying similar arguments as before, it can be seen that the weight of any other vertex in $\overline{T}$ is grater than $\omega_{\overline{T}}(w_1)$. So $C_d(\overline{T}) = C_d(\widehat{T})$. Also, Lemma \ref{lem-cdsc} implies that $v_2\in S_c(\overline{T})$. Therefore, 
$$d_{\overline{T}}{(C_d,S_c)} \geq d_{\widehat{T}}{(C_d, S_c)}.$$
Let $\bar{C}_1, \bar{C}_2, \ldots, \bar{C}_s$ be the branches at $w_1$ in $\overline{T}$, where $\bar{C}_1$ contains $v_2$. Each $\bar{C}_i$, $i\neq 1$, is a path attached to $w_1$. If $s=2$, then $\overline{T}$ is a path-star tree and we are done. Otherwise, again apply Lemma \ref{lem-cdsc} continuously (taking $y=w_1$ and $z$ the pendant vertices in $\bar{C}_i$, $i\neq 1$) to transform $\overline{T}$ to a new tree $T'$ such that there are exactly two btanches in $T'$ at $w_1$, one is $\bar{C}_1$ containing $v_2$ and the other one is a path. Then $T'$ is a path-star tree and $v_2\in S_c(T')$ by Lemma \ref{lem-cdsc}. Note that $C_d(T')$ remains same or move away from $S_c(\overline{T})$. Therefore, $d_{T'}(C_d,S_c)\geq d_{\overline{T}}{(C_d, S_c)}$. This completes the proof. 
\end{proof}

The position of the centroid of a path-star tree $P_{n - g, g}$ can be expressed in terms of $g$. The following result is straight-forward.

\begin{theorem}
\label{remark4}
The centroid of the path-star tree $P_{n - g, g}$ is given by
\begin{equation*}
C_d(P_{n - g, g}) =
\begin{cases}
\begin{cases}
\{\frac{n + 1}{2}\}, &\text{if $g \leq \frac{n - 1}{2}$}\\
\{n - g\}, &\text{if $g > \frac{n - 1}{2}$}
\end{cases}, &\text{if $n$ is odd,}\\
\begin{cases}
\{\frac{n}{2}, \frac{n}{2} + 1\}, &\text{if $g \leq \frac{n}{2} - 1$}\\
\{n - g\}, &\text{if $g > \frac{n}{2} - 1$}
\end{cases}, &\text{if $n$ is even.}
\end{cases}
\end{equation*}
\end{theorem}

Using Theorems \ref{remark1} and \ref{remark4}, we prove the following.

\begin{theorem}\label{remark5}
The distance between the centroid and the subtree core of the path-star tree $P_{n-g, g}$ is given by the following:
If $n$ is odd, then
\begin{equation*}
d_{P_{n - g, g}} (C_d,S_c)=
\begin{cases}
\begin{cases}
\frac{2^g - g - 1}{2}, &\text{if $n - g$ is even}\\
\frac{2^g - g - 2}{2}, &\text{if $n - g$ is odd}
\end{cases}, &\text{if $2^g + 1 \leq n - g$}\\
\begin{cases}
\frac{n - 1}{2} - g, &\text{if $g \leq \frac{n-1}{2}$}\\
0, &\text{if $g > \frac{n-1}{2}$}
\end{cases}, &\text{if $2^g + 1 > n - g.$}
\end{cases}
\end{equation*}
If $n$ is even, then
\begin{equation*}
d_{P_{n - g, g}}(C_d,C_c) =
\begin{cases}
\begin{cases}
\frac{2^g - g - 2}{2}, &\text{if $n - g$ is even}\\
\frac{2^g - g - 3}{2}, &\text{if $n - g$ is odd}
\end{cases}, &\text{if $2^g + 1 \leq n - g$}\\
\begin{cases}
\frac{n}{2} - 1 - g, &\text{if $g \leq \frac{n}{2} -1$}\\
0, &\text{if $g > \frac{n}{2} -1$}
\end{cases}, &\text{if $2^g + 1 > n - g.$}
\end{cases}
\end{equation*}
\end{theorem}

\begin{proof}
First assume that $2^g + 1 \leq n - g$. Then $2g < 2^g +g\leq n-1$ and so $g < \frac{n-1}{2}$. This gives $g\leq \frac{n}{2} -1$ if $n$ is even. Thus, if $n$ is odd, then
\begin{equation*}
d_{P_{n - g, g}}(C_d,S_c) =
\begin{cases}
\frac{n - g + 2^g}{2} - \left(\frac{n + 1}{2}\right), &\text{if $n - g$ is even}\\
\frac{n - g - 1 + 2^g}{2} - \left(\frac{n + 1}{2}\right), &\text{if $n - g$ is odd}.
\end{cases}
\end{equation*}
and if $n$ is even, then
\begin{equation*}
d_{P_{n - g, g}}(C_d,S_c) =
\begin{cases}
\frac{n - g + 2^g}{2} - \left(\frac{n}{2} + 1\right), &\text{if $n - g$ is even}\\
\frac{n - g - 1 + 2^g}{2} - \left(\frac{n}{2} + 1\right), &\text{if $n - g$ is odd}.
\end{cases}
\end{equation*}
Now assume that $2^g + 1 > n - g$. In this case, if $n$ is odd, then
\begin{equation*}
d_{P_{n - g, g}}(C_d,S_c) =
\begin{cases}
n - g - \left(\frac{n + 1}{2}\right), &\text{if $g \leq \frac{n-1}{2}$}\\
0, &\text{if $g > \frac{n-1}{2}$}.
\end{cases}
\end{equation*}
and if $n$ is even, then
\begin{equation*}
d_{P_{n - g, g}}(C_d,S_c) =
\begin{cases}
n - g - \left(\frac{n}{2} + 1\right), &\text{if $g \leq \frac{n}{2} -1$}\\
0, &\text{if $g > \frac{n}{2} -1$}.
\end{cases}
\end{equation*}
Now the theorem follows from the above.
\end{proof}

For a given $n$, we now try to find a $g_0\in\{2,3,\cdots,n-3\}$ for which $P_{n - g_0, g_0}$ will maximize the distance between the centroid and the subtree core among all trees on $n$ vertices.

\begin{proposition}\label{propn6}
For a given $n\geq 5$, let $g_0$ be the smallest integer in $\{2,3,\cdots,n-3\}$ satisfying $2^{g_0} + g_0 > n -1$. Then the path-star tree $P_{n - g_0, g_0}$ maximizes the distance between the centroid and the subtree core among all trees on $n$ vertices.
\end{proposition}

\begin{proof}
By Theorem \ref{thm-cdsc}, we need to show the following two inequalities:
$$d_{P_{n - g_0, g_0}}(C_d, S_c) \geq d_{P_{n - g_0 + k, g_0 - k}}(C_d, S_c) \mbox{ for } k \in \{1, 2, \ldots, g_0 - 2\}$$
and
$$d_{P_{n - g_0, g_0}}(C_d, S_c) \geq d_{P_{n - g_0 - l, g_0 + l}}(C_d, S_c) \mbox{ for } l \in \{1, 2, \ldots, n - g_0 - 3\}.$$

For $g\in\{2,3,\cdots,n-3\}$, if  $2^g + g \leq n - 1$, then $2g < 2^g +g$ implies $g < \frac{n-1}{2}$, that is, $g < \lfloor{\frac{n}{2}}\rfloor$. Thus if $g \geq \lfloor{\frac{n}{2}}\rfloor$, then $2^g + g > n - 1$. So $g_0\leq \lfloor{\frac{n}{2}}\rfloor$ by the definition of $g_0$. Suppose that $g_0 = \lfloor{\frac{n}{2}}\rfloor $. If $n$ even, then $g_0 = \frac{n}{2}$ and so
$$2^{\frac{n}{2} - 1} + \frac{n}{2} - 1 = 2^{g_0 - 1} + g_0 - 1\leq n - 1,$$
which gives $2^{\frac{n}{2}} \leq n$. This is possible only when $n$ is $2$ or $4$. But our assumption is that $n\geq 5$. If $n$ is odd, then $g_0 = \frac{n - 1}{2}$ and so
$$2^{\frac{n - 3}{2}} + \frac{n - 3}{2}= 2^{g_0 - 1} + g_0 - 1 \leq n - 1,$$
which gives $2^{\frac{n - 3}{2}} \leq \frac{n + 1}{2}$. This is possible only when $n$ is $5$ or $7$. For these two values of $n$, it can be verified that $d_{P_{n - g_0, g_0}}(C_d,S_c) \geq d_{P_{n - g, g}}(C_d,S_c)$ for all possible values of $g$.

So assume that $g_0 < \lfloor{\frac{n}{2}}\rfloor$. Note that, for $k \in \{1, 2, \ldots, g_0 - 2\}$,
$$2^{g_0 - k} - (g_0 - k) - [2^{g_0 - k - 1} - (g_0 - k - 1)] = 2^{g_0 - k - 1} - 1 > 0.$$
This implies that
\begin{equation}
\label{eq55}
2^{g_0 - 1} - g_0 + 1 \geq 2^{g_0 - k} - g_0 + k ,
\end{equation}
for all $k \in \{1, 2, \ldots, g_0 - 2\}$. If $n$ is odd, then using (\ref{eq55}), we get
\begin{equation*}\label{eq56}
\begin{aligned}
d_{P_{n - g_0, g_0}}(C_d, S_c) - d_{P_{n - g_0 + k, g_0 - k}}(C_d, S_c) &\geq \frac{n - 1}{2} - g_0 - \left[\frac{2^{g_0 - k} - g_0 + k - 1}{2}\right]\\
&\geq \frac{n - 1}{2} - g_0 - \left[\frac{2^{g_0 - 1} - g_0 + 1 - 1}{2}\right] \\
&= \frac{n - 1 - g_0 - 2^{g_0 - 1}}{2}
\end{aligned}
\end{equation*}
Similarly, if $n$ is even, then
\begin{equation*}\label{eq57}
\begin{aligned}
d_{P_{n - g_0, g_0}}(C_d, S_c) - d_{P_{n - g_0 + k, g_0 - k}}(C_d, S_c) &\geq \frac{n}{2} - 1 - g_0 - \left[\frac{2^{g_0 - k} - g_0 + k - 2}{2}\right]\\
&\geq \frac{n}{2} - 1 - g_0 - \left[\frac{2^{g_0 - 1} - g_0 + 1 - 2}{2}\right]\\
&= \frac{n - 1 - g_0 - 2^{g_0 - 1}}{2}
\end{aligned}
\end{equation*}
Thus, in both case,
\begin{equation*}\label{eq57}
\begin{aligned}
d_{P_{n - g_0, g_0}}(C_d, S_c) - d_{P_{n - g_0 + k, g_0 - k}}(C_d, S_c) &\geq \frac{n - 1 - g_0 - 2^{g_0 - 1}}{2}\\
&= \frac{n - 1 - \left[2^{g_0 - 1} + g_0 - 1 + 1\right]}{2}\\
&\geq \frac{n - 1 - (n - 1 + 1)}{2}\\
&= \frac{-1}{2}
\end{aligned}
\end{equation*}
Since the left hand side is an integer, we must have
$$d_{P_{n - g_0, g_0}}(C_d, S_c) - d_{P_{n - g_0 + k, g_0 - k}}(C_d, S_c) \geq 0.$$

We now show that $d_{P_{n - g_0, g_0}}(C_d, S_c) \geq d_{P_{n - g_0 - l, g_0 + l}}(C_d, S_c)$ for $l \in \{1, 2, \ldots, n - g_0 - 3\}$. If $n$ is odd, then
$$d_{P_{n - g_0, g_0}}(C_d, S_c) - d_{P_{n - g_0 - l, g_0 + l}}(C_d, S_c) \geq \frac{n - 1}{2} - g_0 - \frac{n - 1}{2} + g_0 + l = l.$$
Similar argument holds if $n$ is even. This completes the proof.
\end{proof}

As a consequence of Proposition \ref{propn6}, we have the following.

\begin{theorem}
Let $T$ be tree on $n\geq 5$ vertices and let $g_0$ be the smallest positive integer such that $2^{g_0} + g _0> n - 1.$ Then $d_T(C_d,S_c)\leq \lfloor \frac{n-1}{2} \rfloor -g_o.$
\end{theorem}

\section{ Position of the centroid}

In this section, we study the position of the centroid with respect to the center and the subtree core of a tree. We prove that the centroid always lies in the path connecting the center and the subtree core in any path-star tree. However, this statement need not be true for a general tree.

\begin{proposition}\label{propn7}
In any path-star tree $P_{n - g, g}$, the centroid lies in the path connecting the center and the subtree core.
\end{proposition}

\begin{proof} By Theorem \ref{remark2},  the vertex $\lceil{\frac{n - g + 1}{2}}\rceil \in C(P_{n - g, g})$. Again, by Theorem \ref{remark4}, $C_d(P_{n-g,g})=\{n-g\}$ or the vertex $\lceil{\frac{n}{2}}\rceil \in C_d(P_{n - g, g})$. First assume that $C_d(P_{n-g,g})=\{n-g\}$. This happens only when $g> \lfloor\frac{n-1}{2}\rfloor$. Then
$$2^g +1> 2^{\lfloor\frac{n-1}{2}\rfloor}+1 > n- \lfloor\frac{n-1}{2}\rfloor > n-g.$$
So, by Theorem \ref{remark1}, $S_c(P_{n - g, g})=\{n-g\}$. As $n-g>\lceil{\frac{n - g + 1}{2}}\rceil$ so, the statement follows.

Now assume that $\lceil{\frac{n}{2}}\rceil \in C_d(P_{n - g, g})$. This happens only when $g\leq \lfloor\frac{n}{2}\rfloor$. Since $\lceil{\frac{n - g + 1}{2}}\rceil \leq \lceil{\frac{n}{2}} \rceil$ for $g \in\{2,3,\cdots,
\lfloor\frac{n}{2}\rfloor\}$, the center $C(P_{n - g, g})$ is contained in the branch at  $\lceil{\frac{n}{2}}\rceil$ which contains the vertex $\lceil{\frac{n}{2} }\rceil - 1$. If $2^g +1>n-g$, then $S_c(P_{n - g, g})=\{n-g\}$ and the statement follows. If $2^g +1\leq n-g$, then $\lceil \frac{n-g+2^g}{2}\rceil\in S_c(P_{n - g, g})$. Since $\lceil \frac{n-g+2^g}{2}\rceil > \lceil{\frac{n}{2}}\rceil$, $S_c(P_{n - g, g})$ is contained in the branch at $\lceil{\frac{n}{2}}\rceil$ which contains the vertex $n-g$ and the statement follows.
\end{proof}

We now give an example of a tree in which the centroid does not lie in the path connecting the center and the subtree core.

\begin{figure}[!ht]
 \includegraphics[scale=1.2]{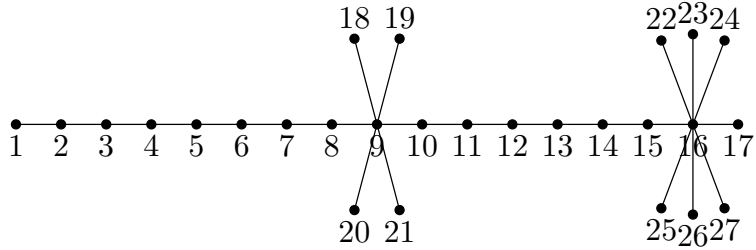}
\caption{Centroid outside the path connecting center and subtree core}\label{fig:4}
\end{figure}

\begin{example} Consider the tree $T$ in Figure \ref{fig:4}. Observe that $C(T) = \{9\},$ and $C_d(T) = \{10\}$. We will show that $S_c(T) = \{9\}$.
Let $B_1, B_2$ be the two components of $T-\{8,9\}$ (deleting the edge $\{8,9\}$ from $T$) containing vertices $8$ and $9,$ respectively.  Then
\begin{equation}
\begin{aligned}
f_T(9) - f_T(8) &= f_{B_2}(9) + f_{B_2}(9)f_{B_1}(8) - f_{B_1}(8) - f_{B_2}(9)f_{B_1}(8)\\
&= f_{B_2}(9) - f_{B_1}(8) >0.
\end{aligned}
\end{equation}
The last inequality holds since $B_2$ contains a copy of $B_1$ (by identifying the vertex $8$ of $B_1$ with $9$ of $B_2$) and $B_2$ has more vertices than $B_1$. So $f_T(9 )> f_T(8)$. Let $M$ and $N$ be the two components of $T-\{9,10\}$ containing vertices $9$ and $10,$ respectively. Then
\begin{equation}
\begin{aligned}
f_T(9) - f_T(10) &= f_{M}(9) + f_{M}(9)f_{N}(10) - f_N(10) - f_{M}(9)f_{N}(10)\\
&= f_{M}(9) - f_N(10)\\
&= (9 \times 2^4) - (6 +2^7)\\
&= 144 - 134>0.
\end{aligned}
\end{equation}
So $f_T(9) > f_T(10)$ and hence $S_c(T) = \{9\}$. Thus $C_d(T)$ does not lie on the path connecting $C(T)$ and $S_c(T)$.
\end{example}

\vskip 1cm

\noindent{\bf Address}:\\
School of Mathematical Sciences\\
National Institute of Science Education and Research (HBNI), Bhubaneswar\\
P.O.- Jatni, District- Khurda, Odisha-752 050, India\\
E-mails: dheer.nsd@niser.ac.in, klpatra@niser.ac.in
\end{document}